\documentclass[12pt]{amsart}
\usepackage[latin1]{inputenc} 
\usepackage{mathtools}
\usepackage{graphics}
\usepackage[top=1in, bottom=1in, left=1in, right=1in,footskip=.5in]{geometry}
\usepackage{amssymb,amsmath,amsthm,amscd,epsf,latexsym,verbatim,graphicx,amsfonts} 
\usepackage{enumitem}
\usepackage{mathrsfs}
\usepackage{tikz-cd}
\usepackage{xargs}
\usepackage{comment}
\usepackage{hyperref}
\usepackage{thmtools}
\usepackage{bm}
\usepackage[capitalise,noabbrev]{cleveref}
\usepackage{mathabx}
\usepackage[backend=biber,style=numeric,maxbibnames=99]{biblatex}
\usepackage{subcaption}
\usepackage{setspace}
\usepackage{lineno}

\usepackage{etoolbox}
\newtheoremstyle{definition-nobrackets}
  {\topsep}   
  {\topsep}   
  {\upshape}  
  {0pt}   	
  {\bfseries} 
  {.}     	
  {5pt plus 1pt minus 1pt} 
  {\thmname{#1}\thmnumber{ #2}\normalfont\thmnote{ #3}} 

\theoremstyle{definition-nobrackets}

\newcounter{Theorem}[section]
\numberwithin{Theorem}{section}
\numberwithin{equation}{section}
\newtheorem{theorem}[Theorem]{Theorem}
\newtheorem{example}[Theorem]{Example}
\newtheorem{definition}[Theorem]{Definition}
\newtheorem{corollary}[Theorem]{Corollary}

\newtheorem{lemma}[Theorem]{Lemma}
\newtheorem{remark}[Theorem]{Remark}

\newcommand{\wt}[1]{\widetilde{#1}}
\newcommand{\Z}{\mathbb{Z}}

\newcommand{\affineS}{\wt{S}}
\newcommand{\asg}[1]{\wt{S}_{#1}}
\newcommand{\floor}[1]{\lfloor #1 \rfloor}

\newcommand{\defn}[1]{\textbf{\emph{#1}}}

\DeclareMathOperator{\Inv}{Inv}

\title[A Triangle Characterization of
Affine Permutation Inversion Graphs]{A ChatGPT-assisted Triangle Characterization of \\
Affine Permutation Inversion Graphs}
\author{Sara C. Billey}
\address{Department of Mathematics, University of Washington, Seattle,
WA, USA}
\email{billey@uw.edu}

\author{Herman Chau}
\address{Department of Mathematics, University of Washington, Seattle,
WA, USA}
\email{beijingherman@gmail.com}

\author{Kevin Liu}
\address{Department of Mathematics and Computer Science, The University of
the South, Sewanee, TN, USA}
\email{keliu@sewanee.edu}

\date{}



\begin{document}

\begin{abstract}
Inversion sets of permutations in the affine symmetric group
$\widetilde{S}_n$ were studied extensively by Bj\"orner and
Brenti. One of their methods for encoding an inversion set is through
an affine inversion graph, which is a certain weighted graph on vertex
set $[n]=\{1,2,\ldots,n\}$. Subsequent work by Papi characterized
which graphs arise as affine inversion graphs. In this paper, we
provide an alternative characterization in terms of a simple local
condition on each triangle in a weighted tournament graph.  This new
characterization was produced with the assistance of ChatGPT, which
suggested several key insights that simplified portions of Papi's
original characterization. Consequences of our characterization
include efficient algorithms for recognizing inversion graphs and
inversion sets. Furthermore, we give bounds on the weights along
directed paths, and we show that standardizing the labels on an
induced subgraph results in another inversion graph.  We conclude with
a new order $O(|R|+n^{3})$ algorithm for testing if a given set $R$ is
the inversion set of an affine permutation.   
\end{abstract}

\maketitle

\section{Introduction}\label{sec:graphs}

Artificial intelligence in mathematics is an area of growing interest. 
Early uses included generating conjectures  \cite{davies2021,HJCBRBK.2025},
some of which were later proven by mathematicians, e.g., see \cite{gurevich2024}. 
Models now exist for specific subfields in mathematics, and these 
models are capable of improving previously best-known results, 
constructing new examples distinct from previous ones, and 
resolving long-standing open problems \cite{charton2024,swirszcz2025}. 
Recently, an internal model at OpenAI produced a counterexample to Erd\"os's 
longstanding unit distance conjecture \cite{alon2026,openai2026}.  

This paper involves a ChatGPT-assisted solution to a problem about 
the affine symmetric group $\widetilde{S}_n$. This group appears in
numerous contexts in mathematics, including combinatorics,
representation theory, and Lie algebras. See \cite{Lewis} for a
general overview and the references therein for further details. The
group $\widetilde{S}_n$ can be described via regions in an infinite
affine hyperplane arrangement, as a Coxeter group, and through certain
$n$-periodic bijections on $\mathbb{Z}$. This paper will focus on the
latter realization as bijections, and hence, we refer to elements in
$\widetilde{S}_n$ as affine permutations.  Similar to the classical
permutations in the symmetric group $S_n$, affine permutations in
$\widetilde{S}_n$ have finite inversion sets, which record pairs that
are out of order.  This work was motivated by our ongoing
investigations on more general inversion sets and the affine higher Bruhat orders
\cite{BilleyChauLiu.2026.affineHB,billeychauliu.fpsac2026,Chau.2025}.

In \cite{bjorner_brenti_1996}, Bj\"orner and Brenti defined an acyclic
directed multigraph encoding the inversions of an affine permutation
$w\in \asg{n}$ in a compact form, and they showed that this graph
uniquely determines an affine permutation. They posed the question of
characterizing all such affine inversion graphs.  This problem was
solved by Papi \cite{Papi.1997}.  We observed experimentally
that in many cases a simpler characterization exists using certain
types of local rules, and we proved several lemmas supporting
a characterization in terms of these local rules. Although our original
guess was insufficient, it led us to use ChatGPT to refine it into 
an elegant characterization, which is the main result of this 
paper. 

To state the characterization, we use the following notation from
\cite{Papi.1997}. Further notation and definitions used can be found
in \Cref{sec:background}.  Let $G$ be a weighted directed graph on
$[n]:=\{1,2,\dotsc , n \}$ with adjacency matrix $t(\cdot,\cdot) \in
\mathbb{Z}_{\geq 0}^{n\times n}$.  Thus, if $a\to b$ is an edge in
$G$, then $t(a,b)$ is the weight on this edge.  If $a\to b$ is not an
edge in $G$, then $t(a,b)=0$ by definition.  Define the \defn{shifted
adjacency matrix} $\widetilde t(a,b)$ by
\begin{equation}\label{eq:shiftdef}
\widetilde t(a,b):=
\begin{cases}
 t(a,b), & a\leq b,\\
 t(a,b)-1, & a>b.
\end{cases}
\end{equation}

A \defn{tournament} is a directed graph with exactly one edge
between each pair of distinct vertices.  Furthermore, a tournament 
$T$ on vertex set $[n]$ with
nonnegative integer edge weights $t(a, b)$ for all $a,b\in [n]$
satisfies the \defn{zero weight condition} if whenever
$t(a,b)=t(b,a)=0$ for $1\leq a<b\leq n$, then the edge connecting
$a,b$ is oriented $a \to b$.

\begin{theorem}[(Triangle Characterization)]\label{thm:sufficiency}
Let $T$ be a weighted tournament on $[n]$ with adjacency matrix
$t(\cdot,\cdot)\in \Z_{\geq 0}^{n \times n}$ satisfying the zero
weight condition. Let $\widetilde t(\cdot, \cdot)$ be the shifted adjacency
matrix of $T$ as defined in \eqref{eq:shiftdef}.  Then, there exists a
unique $w\in\widetilde S_n$ with affine inversion graph $G_w=T$ if and
only if 
\begin{equation}\label{eq:shiftedtriangle}
\widetilde t(a,c) - \widetilde t(a,b)-\widetilde t(b,c)\ \in\ \{0,1\}
\end{equation}
for every $a\to b\to c$ in $T$.
\end{theorem}

We refer to the condition on all length 2 paths in
\eqref{eq:shiftedtriangle} as the \defn{Boolean triangle condition}.
In general, a tournament is not necessarily acyclic, but if it
satisfies the zero weight condition and the Boolean triangle
condition, we will show in \Cref{lem:transitivity} that it must be
acyclic as well.  These two conditions also lead to the following
corollary.

\begin{corollary}[(Local recognition complexity)]\label{cor:verification.time}
 Let $G$ be a weighted directed graph on $[n]$. 
One can verify whether $G$ is realizable as an
affine inversion graph in $O(n^3)$ arithmetic operations.
Furthermore, if $G$ is nonrealizable, then it has a certificate of failure consisting
 of at most 3 vertices.
\end{corollary}

In \Cref{sec:background}, we state the required background for the
proof of the Triangle Characterization.  The proof appears in
\Cref{sec:triangle.rule} after proving several lemmas.  Several consequences
 appear at the end of the paper.  These include
path inequalities, a hereditary property, and a new algorithm to
recognize an affine inversion set.

\subsection*{Acknowledgments and use of AI tools} The authors used
ChatGPT Pro on February 9, 2026, to generate a candidate statement and
proof outline for Theorem~\ref{thm:sufficiency}.  The input included
Papi's paper \cite{Papi.1997}, \Cref{lem:m.simplified},
\Cref{lem:triangle.identity}, and the hope that a variation on the
condition $t(a,c) - t(a,b)- t(b,c)\ \in\ \{0,1,-1\}$ for every
directed path $a\to b\to c$ in $T$ would suffice to characterize the
affine inversion graphs.  The response took 54 minutes. It included a
counterexample using the acyclic tournament $1 \to 2 \to 3$, $1 \to 3$
with all edge weights 1.  Such a graph would imply the window notation
for an affine permutation $(w_{1},w_{2},w_{3})$ has
$w_{1}<w_{2}<w_{3}$ because of the edge directions. In addition,
the edge weights of 1 require that $w_{2}-w_{1}>n$, \
$w_{3}-w_{2}>n$, and $w_{3}-w_{1}<2n$, which is impossible.  The
response also proposed the Boolean triangle condition given in
\Cref{thm:sufficiency}, along with other conditions.  We have improved
upon both the initial proposed characterization and proof in this
paper.

\section{Background}\label{sec:background}

A bijection $w: \Z \to \Z$ is called \defn{$n$-periodic} if for all $x
\in \Z$, we have $w(x+n) = w(x)+n$.  The
\defn{affine symmetric group} $\affineS_n$ is the group of all
$n$-periodic bijections $w: \Z \to \Z$ that satisfy the additional
property $\sum_{a \in [n]}w(a) = \binom{n+1}{2}$.  Notice that since
affine permutations are $n$-periodic, an element $w \in \affineS_n$ is
uniquely defined by its values on $[n]=\{1,\ldots,n\}$. The \defn{window notation} of 
$w$ expresses these values in the form $(w(1),w(2),\ldots,w(n))$.

For $w\in \affineS_n$, the \defn{inversion set} of $w$ is
\begin{equation}\label{eq:inversion.set.def}
\Inv(w)=\left\{(a,b)\in [n] \times \mathbb{Z} : a<b \text{ and } w(a)>
w(b) \right\}.
\end{equation}
Define the \defn{length} of $w$, denoted $\ell(w)$, to be the number
of inversions of $w$.  The length of an affine permutation is always
finite.

Affine permutations are uniquely determined by their inversion set,
just as classical permutations are uniquely determined by their
inversion set \cite[Thm. 4.6]{bjorner_brenti_1996}.  Furthermore,
inversion sets of affine permutations can be completely characterized
as follows.  This presentation follows from the discussion of affine
inversions in \cite[Prop. 2.1]{Barkley.Speyer.2024} and the more
general characterization of inversion sets for Coxeter groups in
\cite[Lemma 4.1(d)]{Dyer.2019}, though stated a bit differently. Here
we consider each $(a,b)\in \Inv(w)$ under its natural identification
in the quotient group $\mathbb{Z}^2/(\mathbb{Z}\cdot (n,n))$.

\begin{theorem}\label{lem:inversion.sets} Let $n$ be a positive
integer, and let $R$ be a finite subset of $\mathbb{Z}^2/(\mathbb{Z}\cdot (n,n))$.
Then, $R$ is the inversion set for some affine permutation in
$\affineS_{n}$ if and only if for all triples of integers
$a<b<c$, we have
\begin{enumerate}[label=(\Roman*), ref=(\Roman*)]
\item \label{part:trapped.value}$(a,c) \in R$ implies  $(a,b) \in R$ or $(b,c) \in R$, 
\item \label{part:inv.packet.gaps.implies.all}$(a,b) \in R$ and $(b,c) \in R$ implies  $(a,c) \in R$, and
\item \label{part:inv.congruence}$(a,b) \in R$ implies $a \not \equiv b \pmod{n}$ and $(a,b-mn)
\in R$ for all $m \in \Z_{\geq 0}$ such that $ a < b-mn<b$.
\end{enumerate}
\end{theorem}

From \Cref{lem:inversion.sets}, it follows that if $(a,b) \in \Inv
(w)$, then $(b,a+kn) \not \in \Inv (w)$ for all $k \in \Z$.
Furthermore, for each $(a,b) \in \Inv (w)$ with $b-a<n$, there exists
a value $m_{a,b}$ such that $(a,b+kn) \in \Inv (w) $ if and only if
$0\leq k<m_{a,b}$.  Thus, inversion sets of affine permutations can be
encoded succinctly by certain directed acyclic graphs with edge
weights encoding the nonnegative integer values $m_{a,b}$.  As
mentioned in the introduction, these graphs were originally defined by
Bj\"orner and Brenti \cite{bjorner_brenti_1996}, prior to
\cite{Barkley.Speyer.2024,Dyer.2019}.  Recall the floor of $x$,
denoted $\floor{x}$, is the largest integer less than or equal to $x
\in \mathbb{R}$, even if $x<0$.

\begin{definition}\cite[\S 5]{bjorner_brenti_1996}\label{def:BB-graph}
For \(w\in \widetilde S_n\), define the \defn{(completed) affine
inversion graph} \(G_{w}\) to be the weighted tournament on \([n]\)
whose edge between \(a<b\) is oriented \(a\to b\) if \(w_a<w_b\), and
\(b\to a\) if \(w_b<w_a\).  Either way, the weight on this edge is
\[
\left|\left\lfloor \frac{w_b-w_a}{n}\right\rfloor\right|.
\]
\end{definition}

\begin{remark}\label{rem:zero_weight_condition}
We note that $\left| \left\lfloor (w_{b}-w_{a})/n \right\rfloor
\right| = 0$ if and only if $0<w_{b}-w_{a}<n$.  Therefore, one may
omit the directed edges of weight $0$ with no loss of information.
Bj\"orner-Brenti chose to omit such edges, but we prefer to keep these
zero weight edges in place, which results in $G_w$ being a weighted
tournament satisfying the zero weight condition. This simplifies the 
statement of \cref{thm:sufficiency} and 
reduces casework in its proof. 
\end{remark}

\begin{theorem}\cite[Thm. 4.6,
Cor. 5.2]{bjorner_brenti_1996}\label{thm:bb.unique}
The affine permutation $w$ is uniquely determined from the list of
weighted in-degrees of the vertices in $G_w$, where each edge contributes
its weight to the in-degree of its target.
\end{theorem}

\begin{example} 
Consider $w=(4,5,-1,2)\in \affineS_{4}$, with inversion set
\[
\Inv (w)=\{(1, 3), (1, 4), (1, 7), (2, 3), (2, 4), (2, 7)\}.
\]
The affine inversion graph $G_{w}$ is shown in \Cref{fig:example-23}.
Observe the edge from $1\to 2$ is labeled 0 because $4=w_{1}<w_{2}=5$
are in order and within $4$ of each other.  The edge $3 \to 2$ has
label $\left| \left\lfloor \frac{(w_{3}-w_{2})}{4} \right\rfloor
\right| =\left| \left\lfloor \frac{-6}{4} \right\rfloor \right|=2$.
This edge label indicates the number of inversions on the congruence
classes $2,3$ from \Cref{lem:inversion.sets}
\Cref{part:inv.congruence}, namely $(2,3)$ and $(2,7)$. Note that the
edge $(3,2)$ points in the reverse direction from the inversion
$(2,3)$ as defined.  Observe also that there is a directed path $3 \to
4 \to 1 \to 2$. This corresponds to the unique total order on $[n]$
determined by the indices where values in the window $(4,5,-1,2)$ are
in increasing order.
\end{example}

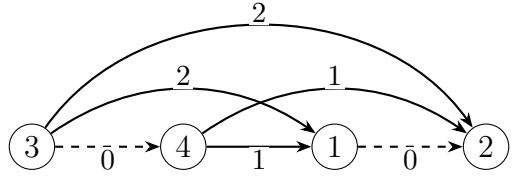
\begin{figure}[ht]
\centering
\begin{tikzpicture}[
    vertex/.style={circle,draw,inner sep=1.5pt,minimum size=6mm},
    edge/.style={-{Stealth[length=2mm]},thick},
    zeroedge/.style={-{Stealth[length=2mm]},thick,dashed},
    lab/.style={font=\small,fill=white,inner sep=1pt}
]
  \node[vertex] (v3) at (0,0) {$3$};
  \node[vertex] (v4) at (2,0) {$4$};
  \node[vertex] (v1) at (4,0) {$1$};
  \node[vertex] (v2) at (6,0) {$2$};

  \draw[zeroedge] (v3) -- node[lab,below] {$0$} (v4);
  \draw[edge]     (v4) to node[lab,below] {$1$} (v1);
  \draw[zeroedge] (v1) -- node[lab,below] {$0$} (v2);

  \draw[edge]     (v3) to[bend left=35] node[lab,above] {$2$} (v1);
  \draw[edge]     (v4) to[bend left=35] node[lab,above] {$1$} (v2);
  \draw[edge]     (v3) to[bend left=55] node[lab,above] {$2$} (v2);
\end{tikzpicture}
\caption{The affine inversion graph $G_w$ for
$w=(4,5,-1,2)\in\widetilde S_4$.  Dashed edges have weight $0$.}
\label{fig:example-23}
\end{figure}

\begin{example}
In \Cref{fig:example-492}, we give a weighted tournament satisfying
the zero-weight condition.  We leave it to the reader to verify 
the affine permutation associated to this graph.
\end{example}

\begin{figure}[ht]
\centering
\begin{tikzpicture}[
    vertex/.style={circle,draw,inner sep=1.5pt,minimum size=6mm},
    edge/.style={-{Stealth[length=2mm]},thick},
    zeroedge/.style={-{Stealth[length=2mm]},thick,dashed},
    lab/.style={font=\small,fill=white,inner sep=1pt}
]
  \node[vertex] (v3) at (0,0) {$3$};
  \node[vertex] (v4) at (2,0) {$4$};
  \node[vertex] (v1) at (4,0) {$1$};
  \node[vertex] (v2) at (6,0) {$2$};

  \draw[edge] (v3) -- node[lab,below] {$1$} (v4);
  \draw[edge] (v4) to node[lab,below] {$1$} (v1);
  \draw[edge] (v1) -- node[lab,below] {$1$} (v2);

  \draw[edge] (v3) to[bend left=35] node[lab,above] {$3$} (v1);
  \draw[edge] (v4) to[bend left=35] node[lab,above] {$2$} (v2);
  \draw[edge] (v3) to[bend left=55] node[lab,above] {$4$} (v2);
\end{tikzpicture}
\caption{The affine inversion graph $G_w$ for
$w=(4,9,-5,2)\in\widetilde S_4$.}
\label{fig:example-492}
\end{figure}
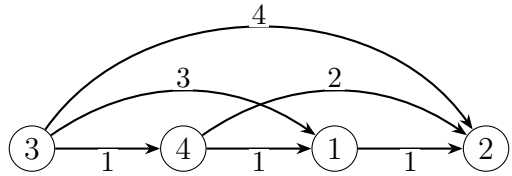

In \cite[\S 5]{bjorner_brenti_1996}, Bj\"orner and Brenti asked if
affine inversion graphs could be characterized directly.  In 1997,
Papi answered this in the affirmative \cite[Thm. 4]{Papi.1997}.  To
state Papi's characterization, recall the shifted weight matrix
$\widetilde t(a,b)$ from \eqref{eq:shiftdef}, given by
\begin{equation}\label{eq:shiftdef.2}
\widetilde t(a,b):=
\begin{cases}
 t(a,b), & a\leq b,\\
 t(a,b)-1, & a>b
\end{cases}
\end{equation}
for any weighted directed graph $G$ on $[n]$ with weighted adjacency
matrix $t(\cdot,\cdot) \in \mathbb{Z}_{\geq 0}^{n\times n}$.

\begin{theorem}\cite[Thm. 4]{Papi.1997}\label{thm:papi4}
Let $T$ be a weighted tournament on $[n]$ with adjacency
matrix $t(\cdot,\cdot)\in \Z_{\geq
0}^{n \times n}$ satisfying the zero weight condition. 
Let $\widetilde t$ be the shifted adjacency matrix as defined in \eqref{eq:shiftdef.2}.
Let $A_{T}$ be the set of positively weighted edges in $T$.  Then 
$T=G_w$ for a necessarily unique $w\in \asg{n}$ if and only if $T$ 
satisfies the following two
conditions.
\begin{itemize}
\item[(I)] For all distinct $a,b,c \in [n]$ such that $(a\to b, b \to
c) \in (A_T\times A_T)$, we have
\[
\widetilde t(a,c)\ \ge\ \widetilde t(a,b)+\widetilde t(b,c). 
\]
\item[(II)] For every edge $(a \to b) \in A_T$, every $c \in
[n]\setminus \{a,b \}$, and every composition $m+r=\widetilde
t(a,b)$ with
\[
m\in
\begin{cases}
\mathbb Z_{>0}, & a<c,\\
\mathbb Z_{\ge 0}, & c<a,
\end{cases}
\qquad
r\in
\begin{cases}
\mathbb Z_{>0}, & c<b,\\
\mathbb Z_{\ge 0}, & b<c,
\end{cases}
\]
one has $\widetilde t(a,c)\ge m$ or $\widetilde t(c,b)\ge r$.
\end{itemize}
\end{theorem}
\begin{remark}
\cref{thm:papi4} has been slightly modified from its original version
to account for our notational conventions.  Similar to Bj\"{o}rner and
Brenti, Papi also deleted edges of weight 0 in $G_w$, and the original
characterization of affine inversion graphs was written with this
convention.  As noted in \cref{rem:zero_weight_condition}, we retain
these edges through the zero weight condition.  Note that Papi uses
$\mathbb N$ for the set of positive integers and
$\mathbb Z_+$ for the set of nonnegative integers in \cite{Papi.1997}. 
\end{remark}

Papi's proof follows directly from a characterization of the
possible inversion sets for affine permutations as affine roots in a
geometric realization of $\affineS_{n}$ using the fundamental
imaginary root.  This characterization appears in his earlier paper
\cite{Papi.1994}.

\begin{remark}
 In addition to Papi's characterization of affine inversion graphs,
one could devise a graph test based on \Cref{lem:inversion.sets}.
However, this would require listing out the candidate inversion set
based on Property (III), and testing each triple for Properties (I) and
(II).  The inversion sets can be arbitrarily large for $\affineS_{n}$,
so this approach is more cumbersome than a direct graphical
characterization.  At the end of \Cref{sec:triangle.rule}, we will
return to the problem of efficient recognition of affine inversion
sets and affine inversion graphs.
\end{remark}

\section{The Boolean Triangle Characterization}\label{sec:triangle.rule}

In order to prove the promised alternative characterization of the
affine inversion graphs, we will use the following notation.  For an
ordered pair of distinct integers $(a,b)$, write $\delta_{a>b}$ for the
indicator function for the inequality, so it is $1$ if $a>b$ and $0$
if $a<b$.  If we need to use $\delta_{a>b}$ in an exponent, we often
abbreviate it to just $a>b$.
With this notation,
\[
(-1)^{a>b}=(-1)^{\delta_{a>b}} \in \{-1,1 \}.  
\]
For any sequence of distinct integers, let
$\mathrm{fl}(z_{1},z_{2},\dotsc , z_{k})$ be the unique permutation of
$[k]$ with the same inversion set.  For example,
$\mathrm{fl}(6314)=4213$.

Assume $w=(w_{1},\dots,w_{n})\in \asg{n}$ is fixed below. For each
$a\in[n]$, let $r_{a}=w_{a}\ \pmod{n}$ and $k_{a}=(w_{a}-r_{a})/n$, so
$w_a=r_a+nk_a$.  We consider $n \pmod n=0$ as usual.

Let $t_{w}$ be the adjacency matrix for the affine inversion graph
$G_{w}$ constructed in \Cref{def:BB-graph}.  Let $\tilde{t}_{w}$ be
the corresponding shifted adjacency matrix from \Cref{eq:shiftdef}.
Combining notation we have 
\begin{equation}\label{eq:t_tilde}
\tilde{t}_{w}(a,b)= t_{w}(a,b) - \delta_{a>b}=\begin{cases}
t_{w}(a,b) &  a\leq b\\
t_{w}(a,b)-1 & a>b.
\end{cases} 
\end{equation}

\begin{lemma}\label{lem:m.simplified}
 If $w_{a}<w_{b}$ for
$a,b \in [n]$, then
\[
t_{w}(a,b) = k_{b}-k_{a} + \frac{(-1)^{r_{a}>r_{b}} - (-1)^{a>b}}{2},
\]
and $t_{w}(a,b)=0$ otherwise.
\end{lemma}

\begin{proof}
The definition of $t_{w}(a,b)$ varies depending on whether $a>b$ or $a<b$. 
However, by using $(-1)^{a>b}$, this can be expressed as
\begin{align*}
t_{w}(a,b)= &\left|\left\lfloor \frac{(-1)^{a>b}(w_{b}-w_{a})}{n} \right\rfloor \right|\\
      =&\left|\left\lfloor \frac{(-1)^{a>b}(r_{b}-r_{a})}{n} +
      (-1)^{a>b}(k_{b}-k_{a}) \right\rfloor \right|.
\end{align*}
 Note $-1<\frac{(-1)^{a>b}(r_{b}-r_{a})}{n}<1$ and $(k_{b}-k_{a})\geq
0$ since $w_{a}<w_{b}$, $a,b\in [n]$, and $w \in \affineS_{n}$. Thus, the signs of
$\frac{(-1)^{a>b}(r_{b}-r_{a})}{n}$ and $(-1)^{a>b}(k_{b}-k_{a})$
depend only on the relative order of $(a,b)$ and $(r_{a},r_{b})$.
Therefore, $t_{w}(a,b)$ is determined by $(k_{b}-k_{a})$ and the relative
orders of $(a,b)$ and $(r_{a},r_{b})$.  It remains to check based on
cases for $(-1)^{a>b}$, $(-1)^{r_{a}>r_{b}}$, and $(k_{b}-k_{a})=0$ or
$(k_{b}-k_{a})>0$, which is left to the reader. The cases involving $(-1)^{r_a>r_b}=-1$
and $(k_b-k_a)=0$ can be omitted, as these cannot occur when the hypothesis
$w_a=r_a+nk_a<r_b+nk_b=w_b$ holds. 
\end{proof}

We now consider values of $t_w$ on triangles in $G_w$. Let
$\binom{[n]}{3}$ denote the size $3$ subsets of $[n]$.  Recall the
notation $\ell(w)$ denotes the number of inversions of $w$ from
\Cref{sec:background}.

\begin{lemma}\label{lem:triangle.identity}
Let $w \in \affineS_{n}$ and $\{a,b,c \} \in \binom{[n]}{3}$ such that
$w_{a}<w_{b}<w_{c}$.  Then, the weighted tournament $G_{w}$ restricted
to vertices $\{a,b ,c \}$ consists of edges $a \to b \to c$ and $a \to
c$, and the triangle identity
\begin{equation}\label{eq:triangle-value}
t_{w}(a,b) + t_{w}(b,c) - t_{w}(a,c) =\frac{(-1)^{\ell(u)} - (-1)^{\ell(v)}}{2} \in \{-1,0,1 \}
\end{equation}
holds, where $u=\mathrm{fl}(r_{a},r_{b},r_{c})$ and
$v=\mathrm{fl}(a,b,c)$.  Equivalently, we can use the following lookup table for
the value of $t_{w}(a,b) + t_{w}(b,c) - t_{w}(a,c)$,
\[
\begin{array}{r|c|c}
 & v \in \{123,231,312 \} & v \in \{213,132,321\}\\
\hline
u \in \{123,231,312 \} & 0 & 1\\
u \in \{213,132,321 \} & -1 & 0
\end{array}
\]
\end{lemma}

\begin{proof}
The restriction of $G_{w}$ follows directly from Definition
\ref{def:BB-graph}.  By Lemma~\ref{lem:m.simplified}, we can express
the triangle sum $t_{w}(a,b) + t_{w}(b,c) - t_{w}(a,c) $ as
\begin{equation}\label{eq:triangle.value.2}
\frac{(-1)^{r_{a}>r_{b}} - (-1)^{a>b}}{2} +
\frac{(-1)^{r_{b}>r_{c}} - (-1)^{b>c}}{2} - \frac{(-1)^{r_{a}>r_{c}} - (-1)^{a>c}}{2}
\end{equation}
after canceling the contributions of $k_{a},k_{b},k_{c}$.  By checking
all permutations $u,v \in S_{3}$, one can observe
\[
(-1)^{r_{a}>r_{b}} + (-1)^{r_{b}>r_{c}} -(-1)^{r_{a}>r_{c}} =
(-1)^{\ell(u)}, 
\]
and 
\[
(-1)^{a>b} + (-1)^{b>c} -(-1)^{a>c} = (-1)^{\ell(v)}.
\]
Hence, after collecting terms in \eqref{eq:triangle.value.2}, we have
\[
t_{w}(a,b) + t_{w}(b,c) - t_{w}(a,c) = \frac{(-1)^{\ell(u)} - (-1)^{\ell(v)}}{2},  
\]
which must be in $\{-1,0,1 \}$ as stated.
\end{proof}

Recall that for $v\in S_3$, $v \in \{213,132,321\}$ if and only if $v$ is an odd
permutation, meaning $(-1)^{\ell(v)}=-1$.  Define the \defn{delta-odd
function} on $v \in S_{3}$ by
\[
\delta_{\operatorname{odd}(v)} = 
\begin{cases}
0& v \in \{123,231,312 \} \\
1 & v \in \{213,132,321\}.
\end{cases}
\]
Observe that for $v=(a,b,c) \in S_{3}$, we have 
\begin{equation}\label{eq:delta_odd}
\delta_{\operatorname{odd}(v)}=\delta_{a>b}
+\delta_{b>c} - \delta_{a>c} \in \{0,1 \}.
\end{equation}
This observation can be used to simplify the triangle identity in
\Cref{lem:triangle.identity}.  It also gives a simple interpretation
for the sum appearing in the Boolean triangle condition from
\eqref{eq:shiftedtriangle}.

\begin{corollary}\label{cor:papi-triangle.identity}
Let $w \in \affineS_{n}$ and $\{a,b,c \} \in \binom{[n]}{3}$ such that
$w_{a}<w_{b}<w_{c}$.  Then, the weighted tournament $G_{w}$ restricted to
vertices $\{a,b,c\}$ is $a \to b \to c$ and $a \to c$, and 
\begin{equation}\label{eq:papi-triangle-value}
\tilde{t}_{w}(a,c) - \tilde{t}_{w}(a,b) - \tilde{t}_{w}(b,c)  =
\delta_{\operatorname{odd}(u)} \in \{0,1 \}
\end{equation}
holds where $u=\mathrm{fl}(r_{a},r_{b},r_{c})$.
\end{corollary}

\begin{proof}
We know from \Cref{def:BB-graph} that $G_{w}$ restricted to vertices 
$\{a,b,c\}$ is $a \to b \to c$ and $a \to c$.
From the triangle identity in \Cref{lem:triangle.identity},
\[
t_{w}(a,b) + t_{w}(b,c) - t_{w}(a,c) = \frac{(-1)^{\ell(u)} - (-1)^{\ell(v)}}{2}
\]
where $u=\mathrm{fl}(r_{a},r_{b},r_{c})$ and $v=\mathrm{fl}(a,b,c)$.
Using \eqref{eq:t_tilde}, we have
\[
\tilde{t}_{w}(a,b) + \delta_{a>b} + \tilde{t}_{w}(b,c) +\delta_{b>c} -
\tilde{t}_{w}(a,c) -\delta_{a>c} = \frac{(-1)^{\ell(u)} - (-1)^{\ell(v)}}{2}.
\]
Rearranging the terms and using \eqref{eq:delta_odd}, 
\begin{equation}\label{eq:delta_odd_lengths}
\begin{split}
\tilde{t}_{w}(a,c) - \tilde{t}_{w}(a,b) - \tilde{t}_{w}(b,c) & = \delta_{a>b}
+\delta_{b>c} - \delta_{a>c} +\frac{(-1)^{\ell(v)} -
(-1)^{\ell(u)}}{2} \\
& = \delta_{\operatorname{odd}(v)}+\frac{(-1)^{\ell(v)} -
(-1)^{\ell(u)}}{2}.
\end{split}
\end{equation}
The fraction on the right is determined by the table
\[
\begin{array}{r|c|c}
 & v \in \{123,231,312 \} & v \in \{213,132,321\}\\
\hline
u \in \{123,231,312 \} & 0 & -1\\
u \in \{213,132,321 \} & 1 & 0
\end{array}
\]
and the result now follows by directly checking that 
$\delta_{\operatorname{odd}(u)}$ matches \eqref{eq:delta_odd_lengths} in each case.
\end{proof}

\begin{lemma}\label{lem:transitivity}
Let $T$ be a weighted tournament on $[n]$ with adjacency matrix
$t(\cdot,\cdot)\in \Z_{\geq 0}^{n \times n}$ satisfying the zero
weight condition and the Boolean triangle condition. Then, $\widetilde
t(a,b)\geq 0$ for all $a \to b$ in $T$, and $T$ is acyclic.
\end{lemma}

\begin{proof}
We first show that if $a\to b$ is an edge of
$T$, then $\widetilde t(a,b)\geq 0.$
Indeed, if $a<b$, then $\widetilde t(a,b)=t(a,b)\geq 0$. If $a>b$, then
\[
\widetilde t(a,b)=t(a,b)-1.
\]
In this case the edge $a\to b$ cannot have weight $0$, since the zero weight
condition would then force the edge between $b<a$ to be oriented $b\to a$.
Hence $t(a,b)\geq 1$, and again $\widetilde t(a,b)\geq 0$.

Next we will show $T$ is acyclic. It is well-known that a tournament
$T$ is acyclic if and only if it does not contain any directed
$3$-cycles.  Suppose for contradiction that $T$ contains a directed
$3$-cycle
\[
a\to b\to c\to a.
\]
Apply the Boolean triangle condition to the directed path $a\to b\to c$. Since
the edge between $a$ and $c$ in $T$ is oriented $c\to a$, we have $t(a,c)=0$, and hence
\[
\widetilde t(a,c)=
\begin{cases}
0, & a<c,\\
-1, & a>c.
\end{cases}
\]
On the other hand, by the first paragraph,
\[
\widetilde t(a,b)\geq 0
\qquad\text{and}\qquad
\widetilde t(b,c)\geq 0.
\]
Therefore,
\[
\widetilde t(a,c)-\widetilde t(a,b)-\widetilde t(b,c)\leq 0.
\]
The Boolean triangle condition says this quantity lies in $\{0,1\}$, so it must
be equal to $0$. In particular, $\widetilde t(a,c)=0$, and therefore $a<c$.

Repeating the same argument for the directed paths $b\to c\to a$ and
$c\to a\to b$ gives
\[
b<a
\qquad\text{and}\qquad
c<b,
\]
respectively. Thus
\[
b<a<c<b,
\]
which is impossible. Hence $T$ contains no directed $3$-cycle, and
therefore $T$ is acyclic.
\end{proof}

\begin{proof}[Proof of \Cref{thm:sufficiency}] By construction, the
affine inversion graph $G_{w}$ is an acyclic weighted tournament with
nonnegative integer edge weights satisfying the zero weight condition.
By \Cref{cor:papi-triangle.identity}, we know that the graph $G_{w}$
satisfies the Boolean triangle condition given in \eqref{eq:shiftedtriangle} for
every $a\to b\to c$ in $G_{w}$ and $w \in \affineS_{n}$.  Hence, it
suffices to prove the converse holds.

Assume $T$ is a weighted tournament on $[n]$ with adjacency matrix
$t(\cdot,\cdot)\in \Z_{\geq 0}^{n \times n}$ satisfying the zero
weight condition and the Boolean triangle condition.  We will verify
conditions (I) and (II) from Theorem~\ref{thm:papi4}.

To verify condition (I), fix distinct $a,b,c\in[n]$ such that 
$(a\to b,b\to c)\in A_T\times A_T$, so $t(a,b)>0$
and $t(b,c)>0$. Then $a\to b\to c$ is a directed path in $T$ by
hypothesis.  By \eqref{eq:shiftedtriangle},
\[
 \widetilde t(a,c) - \widetilde t(a,b) -\widetilde t(b,c)\geq 0,
\]
so in particular
\begin{equation}\label{eq:papi.in.1} 
\widetilde t(a,c)\ge \widetilde t(a,b)+\widetilde t(b,c).
\end{equation}
This is exactly condition~(I) for the pair 
$(a\to b, b \to c)\in A_{T}\times A_{T}$.

Next consider condition (II). Fix an edge $(a\to b)$ of $T$ such that
$t(a,b)>0$.  Let $c\in[n]\setminus\{a,b\}$ and consider the ordered
pairs $(a,c)$ and $(c,b)$.  Crucially observe that $a \to c$ and $c
\to b$ need not lie in $A_{T}$.  If the edge between $a$ and $c$ has
weight zero in $T$, then both matrix entries $t(a,c)$ and $t(c,a)$ are
zero by hypothesis. Consequently $\widetilde t(a,c)$ is either $0$ or
$-1$, depending on the relative order of $a$ and $c$ as integers.
Similarly, $\widetilde t(c,b)=$ is either $0$ or $-1$.

To verify condition~(II) for the given triple $a,b,c$, we will prove
the stronger claim that for every pair of integers $m,r$ such that
$m+r=\widetilde t(a,b)$, one has
\[
\widetilde t(a,c)\ge m \qquad \text{or}\qquad \widetilde t(c,b)\ge r.
\]
To prove the claim, suppose for contradiction that there exist
integers $m$ and $r$ with $m+r=\widetilde t(a,b)$ such that
\[
\widetilde t(a,c)<m
\qquad\text{and}\qquad
\widetilde t(c,b)<r.
\]
Since all quantities are integral, we necessarily have
\[
m\ge \widetilde t(a,c)+1
\qquad \text{and}\qquad
r\ge \widetilde t(c,b)+1.
\]
Thus,
\begin{equation}\label{eq:contradictsum}
m+r\ge \widetilde t(a,c)+\widetilde t(c,b)+2.
\end{equation}
It now suffices to prove the inequality
\begin{equation}\label{eq:upperboundII}
\widetilde t(a,b)\le
\widetilde t(a,c)+\widetilde t(c,b)+1
\qquad\text{for all }c\neq a,b.
\end{equation}
Indeed, combining \eqref{eq:contradictsum} with
\eqref{eq:upperboundII} gives
\[
m+r\ge \widetilde t(a,c)+\widetilde t(c,b)+2
> \widetilde t(a,b),
\]
contradicting $m+r=\widetilde t(a,b)$.

By \Cref{lem:transitivity}, $T$ is acyclic.  Hence, there is a total
order $<_T$ on $[n]$ defined by $a<_T b$ if and only if $a\to b$ is an
edge in $T$.  We analyze the cases depending on the three possible
positions of $c$ in the total order determined by $T$, given that $a
<_T b$.

\begin{itemize}
\item If $a<_T c<_T b$, then $a\to c\to b$ is a directed path in $T$, and
\eqref{eq:shiftedtriangle} gives \eqref{eq:upperboundII} immediately.

\item If $c<_T a$, then $c\to a\to b$ is a directed path in $T$.
Applying \eqref{eq:shiftedtriangle} to $c\to a\to b$ yields
\[
 \widetilde t(c,b)\ \ge\ \widetilde t(c,a)+\widetilde t(a,b),
\]
hence $\widetilde t(a,b)\le \widetilde t(c,b)-\widetilde t(c,a)$.
On the other hand, for any ordered pair $(p,q)\in [n]^{2}$ one has the elementary bound
\begin{equation}\label{eq:pairineq}
\widetilde t(p,q)+\widetilde t(q,p)\ge -1,
\end{equation}
which follows directly from \eqref{eq:shiftdef} and $t(\cdot,\cdot)\ge 0$.
Applying \eqref{eq:pairineq} to $(p,q)=(a,c)$ gives
$-\widetilde t(c,a)\le \widetilde t(a,c)+1$, and therefore
\[
 \widetilde t(a,b)\le \widetilde t(c,b)+\widetilde t(a,c)+1,
\]
which is \eqref{eq:upperboundII}.

\item If $b<_T c$, then $a\to b\to c$ is a directed path in $T$.
Applying \eqref{eq:shiftedtriangle} to $a\to b\to c$ yields
\[
 \widetilde t(a,c)\ \ge\ \widetilde t(a,b)+\widetilde t(b,c),
\]
hence $\widetilde t(a,b)\le \widetilde t(a,c)-\widetilde t(b,c)$.
Applying \eqref{eq:pairineq} to $(p,q)=(b,c)$ gives
$-\widetilde t(b,c)\le \widetilde t(c,b)+1$, and therefore
\[
 \widetilde t(a,b)\le \widetilde t(a,c)+\widetilde t(c,b)+1,
\]
again yielding \eqref{eq:upperboundII}.
\end{itemize}
In all cases, \eqref{eq:upperboundII} holds. As noted above, this results in
a contradiction, so we conclude (II) holds.
\end{proof}

\begin{corollary}[(Path inequalities)] Let $T=G_w$ be an affine
inversion graph, written with shifted weights $\widetilde t$.  If
\[
a_0\to a_1\to\cdots\to a_k
\]
is a directed path in $T$, then $a_{0} \to a_{k}$ is an edge in $T$ as
well, and 
\begin{equation}\label{eq:path.inequalities}
0\le
\widetilde t(a_0,a_k)-
\sum_{r=0}^{k-1}\widetilde t(a_r,a_{r+1})
\le k-1.
\end{equation}
\end{corollary}

\begin{proof}
The claim is trivial for $k=1$, so assume $k\geq 2$. By \cref{lem:transitivity},
 $T=G_w$ is an acyclic tournament, so
the directed path $a_0\to a_1\to\cdots\to a_k$
implies that $a_i\to a_j$ is an edge of $T$ whenever $0\leq i<j\leq
k$.  By \Cref{thm:sufficiency} and the Boolean triangle condition
applied to the directed path $a_0\to a_1\to a_k$, we have
\[
0\leq \widetilde t(a_0,a_k)-\widetilde t(a_0,a_1)-\widetilde
t(a_1,a_k) \leq 1.
\]
By induction applied to the directed path
$a_1\to a_2\to\cdots\to a_k$, we have
\[
0\leq
\widetilde t(a_1,a_k)-
\sum_{r=1}^{k-1}\widetilde t(a_r,a_{r+1})
\leq k-2.
\]
Adding these two inequalities proves \eqref{eq:path.inequalities}.
\end{proof}

\begin{corollary}[(Hereditary property)] Let $G_w$ be an affine
inversion graph on $[n]$, and let $S\subseteq[n]$.  Order-standardize
the vertex set $S$ to $[|S|]$, preserving the natural order of labels,
and keep the induced orientations and weights in the induced subgraph
of $G_w$ on $S$.  The resulting weighted tournament is an affine
inversion graph.
\end{corollary}

\begin{proof}
The zero-weight condition and the Boolean triangle condition are inherited
by induced subtournaments.  The order-standardization preserves
comparisons $a<b$, so the proof follows from \Cref{thm:sufficiency}.
\end{proof}

We return to the problem of characterizing inversion sets via
\Cref{lem:inversion.sets}.  Observe that a priori a direct test of the
conditions in \cref{lem:inversion.sets} requires verification for the
infinite set of triples $a<b<c$ in $\Z$.  However, the three
conditions in \Cref{lem:inversion.sets} can be checked directly in
$O(|R|^2)$ arithmetic operations.  First check the congruence-string
condition \Cref{lem:inversion.sets}\ref{part:inv.congruence}. This
gives, for each ordered residue pair $i\neq j$, a string length
$m_{ij}$ such that
\[
(x,y)\in R
\quad\Longleftrightarrow\quad
x<y,\quad x\not\equiv y\pmod n,\quad
\left\lfloor\frac{y-x}{n}\right\rfloor < m_{\overline{x},\overline{y}}.
\]
Using this membership test, condition~(I) can be checked in $O(|R|^2)$
operations: for each $(a,c)\in R$, if $c-a-1>2|R|-2$, then condition~(I)
fails immediately; otherwise, check all $b$ with $a<b<c$. Condition~(II)
can also be checked in $O(|R|^2)$ operations by running over all ordered
pairs of elements of $R$ that can be translated to the form
$(a,b),(b,c)$ and then checking whether $(a,c)\in R$.

The algorithm below instead compresses the congruence strings into a
weighted tournament and applies the triangle characterization. This
gives an $O(|R|+n^3)$ recognition algorithm.  Since $|R|$ can be
arbitrarily large for a fixed $n$, this improves on the direct
$O(|R|^2)$ test when $n^3\ll |R|^2$.

\begin{theorem}[(Inversion recognition complexity)]\label{cor:triangle.rule.test.inversion}
Let $R$ be a finite subset of $\mathbb{Z}^2/(\mathbb{Z}\cdot (n,n))$.
Then, there is an algorithm requiring at most $O(|R| + n^{3})$
arithmetic operations to test if $R$ is the inversion set of some
necessarily unique $w \in \affineS_{n}$.
\end{theorem}

\begin{proof}
The following algorithm returns a completed affine inversion graph $T$ if $R$
is realizable as an affine inversion set for some $w\in\affineS_n$, and
returns \textsc{False} otherwise.

\begin{enumerate}[
    labelwidth=5.5em,
    labelsep=0.0em,
    leftmargin=!,
    align=left,
    itemsep=0.65em
]

\item[\textbf{Input.}]
A finite set $R\subseteq \Z^2/(\Z\cdot(n,n))$.

\item[\textbf{Output.}]
A completed affine inversion graph $T$, or \textsc{False}.

\item[\textbf{Step 1.}]
\emph{Residue strings.}
For each ordered pair $(a,b)\in[n]\times[n]$, initialize a set
$S_{(a,b)}=\emptyset$. For each class $\rho\in R$, choose a representative
$(x,y)\in\Z^2$.

\begin{enumerate}[
    label=(\alph*),
    leftmargin=4.0em,
    labelsep=0.7em,
    itemsep=0.25em
]
    \item If $x\geq y$ or $x\equiv y\pmod n$, return \textsc{False}.

    \item Otherwise, let $a,b\in[n]$ be the residues of $x$ and $y$ modulo $n$,
    respectively, and set
    \[
    j=\left\lfloor\frac{y-x}{n}\right\rfloor\in\Z_{\geq 0}.
    \]

    \item Insert $j$ into $S_{(a,b)}$.
\end{enumerate}
Observe that with our convention that residues lie in $[n]$, we have
\[
(x,y)=\bigl(a,b+n(j+\delta_{a>b})\bigr)
\]
as elements of $\Z^2/(\Z\cdot(n,n))$.

\item[\textbf{Step 2.}]
\emph{Candidate tournament.}
We try to construct a weighted tournament $T$ on vertex set $[n]$ satisfying
the zero weight condition. For each pair $1\leq a<b\leq n$, first check that
at most one of $S_{(a,b)}$ and $S_{(b,a)}$ is nonempty, and that each
nonempty set among them is an initial segment of the form
\[
\{0,1,\ldots,q-1\}
\]
for some $q\geq 1$. If either check fails, return \textsc{False}.

If the checks pass, assign the edge between $a$ and $b$ by
\[
\begin{array}{rcll}
S_{(a,b)}\neq\emptyset
&\Longrightarrow&
b\to a
&\text{with weight } t(b,a)=|S_{(a,b)}|,\\[2mm]
S_{(a,b)}=\emptyset
&\Longrightarrow&
a\to b
&\text{with weight } t(a,b)=|S_{(b,a)}|.
\end{array}
\]
In the second case the weight may be zero. Thus, if the construction succeeds,
the resulting weighted tournament satisfies the zero weight condition.

\item[\textbf{Step 3.}]
\emph{Triangle test.}
Apply the verification algorithm from \Cref{cor:verification.time} to the
weighted tournament $T$. If $T$ passes the test, return $T$; otherwise return
\textsc{False}. This completes the algorithm.

\end{enumerate}

We now verify correctness.  If the algorithm returns \textsc{False} in
Step~1, then $R$ contains a class that cannot occur in an affine
inversion set, since affine inversions must have representatives
$(x,y)$ with $x<y$ and $x\not\equiv y\pmod n$ by \cref{lem:inversion.sets}.  Note that since
translation by $(n,n)$ preserves both the inequality $x<y$ and the
congruence class of $x-y$ modulo $n$, the test in Step~1 is
independent of the chosen representative.  If the algorithm returns
\textsc{False} in Step~2, then either the congruence-string condition
in \Cref{lem:inversion.sets}\ref{part:inv.congruence} fails, or the
data would force positive weights in both directions between two
vertices. Neither can occur for an affine inversion graph.  If the
algorithm returns \textsc{False} in Step~3, then the candidate
tournament is not a completed affine inversion graph by
\Cref{cor:verification.time}. Hence $R$ cannot be the inversion set of
an affine permutation, since an actual affine inversion set would have
produced exactly its completed affine inversion graph in Step~2.

Conversely, suppose the algorithm returns a tournament $T$.  By
\Cref{thm:sufficiency}, there is a unique $w\in\affineS_n$ such that
$T=G_w$.  We claim that $\Inv(w)=R$. To prove the claim, fix $1\leq
a<b\leq n$. If $S_{(a,b)}=\{0,1,\ldots,q-1\}$, then Step~2 gives the
edge $b\to a$ with weight $q$. Therefore, the inversions of $w$ in
these two residue classes are precisely
\[
(a,b+kn),\qquad 0\leq k<q.
\]
If instead $S_{(b,a)}=\{0,1,\ldots,q-1\}$, then Step~2 gives the edge $a\to b$
with weight $q$. Therefore, the inversions of $w$ in these two residue classes
are precisely
\[
(b,a+(k+1)n),\qquad 0\leq k<q.
\]
Finally, if both $S_{(a,b)}$ and $S_{(b,a)}$ are empty, then the edge between
$a$ and $b$ has weight zero, so there are no inversions in either direction
between these two residue classes. Thus the residue strings recovered from
$G_w=T$ are exactly the residue strings constructed from $R$, and hence
$\Inv(w)=R$.

The construction of the sets $S_{(a,b)}$ and the candidate tournament requires
$O(|R|+n^2)$ operations. The final triangle test requires 
$O(n^3)$ operations by \cref{cor:verification.time},
so the total number of operations is $O(|R|+n^3)$.
\end{proof}

\begin{remark}
In the proof of \Cref{cor:triangle.rule.test.inversion}, one could
also use Papi's affine inversion graph characterization in
\Cref{thm:papi4} instead of a test based on \Cref{thm:sufficiency}.
Condition (I) of \Cref{thm:papi4} requires $O(n^{3})$ operations.  A
naive implementation of condition (II), obtained by enumerating all
compositions $m+r=\widetilde t(a,b)$, requires $O(n^3\cdot M)$ operations, where
$M$ is the maximum value in the weighted adjacency matrix.  Here $M$
can be an arbitrarily large integer, so the corollary above is an
improvement over this approach to verification as well.
\end{remark}

\printbibliography

\end{document}